\documentclass[a4paper,10pt]{scrartcl}

\usepackage{amsmath}
\usepackage{amsfonts}
\usepackage{amssymb}
\usepackage{amsthm}
\usepackage{marvosym}
\usepackage{tikz}
\usetikzlibrary{arrows,topaths}
\usepackage{tkz-berge}
\usepackage{tkz-graph}
\usepackage{enumerate}

\newtheorem{defn}{Definition}
\newtheorem{thm}[defn]{Theorem}
\newtheorem{rem}[defn]{Remark}
\newtheorem{lem}[defn]{Lemma}
\newtheorem{exm}[defn]{Example}

\title{Meso-scale obstructions to\\ stability of 1D center manifolds\\
for networks of coupled differential\\
equations with symmetric Jacobian
}
\author{J.\ Epperlein, A.L.\ Do, T.\ Gross, S.\ Siegmund}

\newcommand{\RR}{{\mathbb R}}

\newcommand{\NN}{{\mathbb N}}
\newcommand{\CC}{{\mathbb C}}

\newcommand{\AC}{{\mathcal A}}
\newcommand{\DC}{{\mathcal D}}
\newcommand{\GC}{{\mathcal G}}
\newcommand{\LC}{{\mathcal L}}
\newcommand{\FC}{{\mathcal F}}
\newcommand{\ones}{{\mathbf 1}}
\newcommand{\setsep}{\:|\:}

\begin{document}

\maketitle

\begin{abstract}
A linear system $\dot x = Ax$, $A \in \RR^{n \times n}$, $x \in
\RR^n$, with $\mathrm{rk} A = n-1$, has a 
one-dimensional center manifold $E^c = \{v \in \RR^n : Av=0\}$. If a
differential equation $\dot x = f(x)$ has a one-dimensional center
manifold $W^c$ at an equilibrium $x^*$ then $E^c$ is tangential to
$W^c$ with $A = Df(x^*)$ and for stability of $W^c$ it is necessary
that $A$ has no spectrum in $\CC^+$, i.e.\ if $A$ is symmetric, it has
to be negative semi-definite.

We establish a graph theoretical approach to characterize
semi-definiteness. Using spanning trees for the graph corresponding to
$A$, we formulate meso-scale conditions with certain principal minors
of $A$ which are necessary for semi-definiteness. We illustrate these
results by the example of the Kuramoto model of coupled oscillators. 
\end{abstract}

\section{Introduction}

Center manifolds of differential equations $\dot x = f(x)$ play an important role, e.g.\ in bifurcation theory \cite{Crawford1991, Sijbrand1985} or, more specifically, if certain quantities or invariants are preserved, e.g.\ for mechanical systems with rigid-body modes, in chemical kinetics, compartmental modeling, and population dynamics, oftentimes these center manifolds are one-dimensional and consist of a continuum of equilibria (see e.g.\ \cite{Aulbach1984, Bhat-Bernstein2003} and the references therein). If a center manifold $W^c$ is one-dimensional and contains an equilibrium $x^*$ then $\mathrm{rk} Df(x^*) = n-1$. For $W^c$ to be stable it is then necessary that $A := Df(x^*)$ has no eigenvalue with positive real part, i.e.\ $A$ has to be negative semi-definite, since we assume the Jacobian $A$ to be symmetric. The assumption $A=A^T$ is e.g.\ satisfied if the coupling of the differential equations $\dot x_i = f(x_1,\dots,x_n)$, $i=1,\dots,n$, is symmetric. W.l.o.g.\ one can assume that $A = (a_{ij})$ has row sum zero, i.e.\ $\sum_{j=1}^n a_{ij} = 0$ for $i=1,\dots,n$. This is a consequence of the fact that the unit vectors $v_1 \in E^c = \{v \in \RR^n  : Av=0\}$ resp.\ $w_1 = \frac{1}{|(1,\dots,1)|} (1,\dots,1) \in \RR^n$ can be extended to orthogonal bases $\{v_1,\dots,v_n\}$ and $\{w_1,\dots,w_n\}$ of $\RR^n$, i.e.\ the matrices $V = (v_1|\dots|v_n)$ and $W = (w_1|\dots|w_n)$ are orthogonal. The orthogonal transformation $y = WV^T x$ preserves stability and transforms $\dot x = Ax$ and its center manifold $E^c$ into $\dot y = WV^T A VW^T y$ with corresponding center manifold $\{\alpha w_1 \in \RR^n : \alpha \in \RR\}$, i.e.\ $A w_1 = 0$. 

In the present paper we discuss semi-definiteness of a matrix $A \in \RR^{n \times n}$ by Sylvester's criterion. Our results provide a rigorous proof of a recently formulated stability condition \cite{Do2012} and show that utilizing spanning trees of the Coates graph of $A$ is a promising approach for identifying obstructions to the stability of one-dimensional center manifolds, or formulated in a more application-oriented language, the impact of meso-scale structural properties on the dynamics of complex networks. As a motivational example consider the Kuramoto model of a network of coupled oscillators (see \cite{PikovskyRosenblumKurths2003} for an excellent survey)
\begin{equation}\label{Kuramoto}
  \dot{\theta}_i=\omega_i+\sum_{j\neq i} B_{ij}\sin(\theta_j-\theta_i)\ , \quad \forall i\in \{1\ldots N\} 
\end{equation}
on the $N$-torus $T^N = (S^1)^N$, with $S^1 = \RR / (2\pi \mathbb{Z}) \simeq [0,2\pi)$ denoting the circle. Here, $\theta_i \in S^1$ and $\omega_i \in S^1$ denote the phase and the intrinsic frequency of node $i$, while $(B_{ij}) \in \RR^{N\times N}$ is the symmetric weight matrix representing an undirected, weighted interaction network. 
Two oscillators $i,j$ are thus connected if $B_{ij}=B_{ji}\neq 0$.
Note that, if for a solution $\theta = (\theta_1,\dots,\theta_N) : \RR \rightarrow T^N$ we sum up all equations in \eqref{Kuramoto}, then  $\sum_{i=1}^N \dot\theta_i(t) = \sum_{i=1}^N \omega_i + \sum_{i=1}^N \sum_{j\neq i} B_{ij}\sin(\theta_j(t)-\theta_i(t))$. Since $B_{ij} = B_{ji}$, it follows that $\sum_{i=1}^N \sum_{j\neq i} B_{ij}\sin(\theta_j(t)-\theta_i(t)) = 0$ for $t \in \RR$ and hence $\sum_{i=1}^N \dot\theta_i(t) = \sum_{i=1}^N \omega_i$. 

The Kuramoto model provides an example for studying synchronisation in networks and is widely accepted as a simple model of continous dynamics on networks. To see that the following three statements are equivalent
\begin{itemize}
  \item[(i)] $\theta^*$ is a \emph{phase-locked} solution, i.e.\ $\dot{\theta}^*_i(t) = \dot{\theta}^*_j(t)$ for all $t \in \RR, i, j \in \{1,\dots,N\}$,
  \item[(ii)] $\theta^*(\cdot + {\omega(T)})$ is a phase-locked solution for each ${\omega(T)} \in \RR$,
  \item[(iii)] $\dot \theta^*_j(t) = \Omega$ for all $t \in \RR$, $j \in \{1,\dots,N\}$, with the \emph{mean frequency} $\Omega := \frac{1}{N} \sum_{i=1}^N \omega_i$,
\end{itemize}
note that $(iii) \Rightarrow (ii) \Rightarrow (i)$ is obvious, and $(i) \Rightarrow (iii)$ follows from $N \dot\theta_j^*(t) = \sum_{i=1}^N \dot\theta_i^*(t) = \sum_{i=1}^N \omega_i^*$.
As a consequence, every phase-locked solution can be re\-presented as $t \mapsto \theta^*(t+{\omega(T)}) = \theta^*({\omega(T)}) + (\Omega t,\dots,\Omega t)$ for an ${\omega(T)} \in \RR$ and a phase-locked solution $\theta^*$.

The Kuramoto model \eqref{Kuramoto} can also be viewed in a `rotating frame', more precisely, in the new coordinates $x_i(t) = \theta_i(t) - \Omega t$. The transformed model
\begin{equation}\label{KuramotoTransformed}
  \dot{x}_i=\omega_i - \Omega + \sum_{j\neq i} B_{ij}\sin(x_j-x_i)\ , \quad \forall i\in \{1\ldots N\} 
\end{equation}
has mean frequency equal to zero and every phase-locked solution $\theta^*(\cdot + {\omega(T)})$, ${\omega(T)} \in \RR$, of \eqref{Kuramoto} is transformed to an equilibrium $x^{({\omega(T)})}$ of \eqref{KuramotoTransformed} with $x^{({\omega(T)})}_i = \theta^*_i({\omega(T)})$, giving rise to a 1-dimensional manifold ${\cal N} = \{x^{({\omega(T)})} \in T^N : {\omega(T)} \in \RR\}$ consisting of a continuum of equilibria of \eqref{KuramotoTransformed}.
The corresponding manifold ${\cal M} = \{\theta^*(t) \in T^N : t \in \RR \textup{ and } \theta^* \textup{ is a phase-locked solution of \eqref{Kuramoto}}\}$, consisting of all orbits of phase-locked solutions of \eqref{Kuramoto}, is called \emph{stable with asymptotic phase} if there exists a neighborhood $U$ of ${\cal M}$ such that for any solution $\theta$ starting in $U$, i.e.\ $\theta(0) \in U$, there exists a phase-locked solution $\theta^*$ in $\cal M$, i.e.\ $\theta^*(t) \in \cal M$ for $t \in \RR$, with $\lim_{t \to \infty} d(\theta(t), \theta^*(t)) = 0$, where $d(\cdot,\cdot)$ denotes the distance between two points on the torus $T^N$.

The linearization of \eqref{KuramotoTransformed} at an equilibrium $x^{({\omega(T)})} \in {\cal N}$ is given by $\dot{x} = A x$ with
\begin{equation}\label{Jacobian}
  A = 
  \begin{pmatrix}
    - \sum_{j\neq 1} B_{1j}\cos(\varphi_{j1}) & B_{12}\cos(\varphi_{21}) & \dots & B_{1N}\cos(\varphi_{N1})
  \\
    B_{21}\cos(\varphi_{12}) & - \sum_{j\neq 2} B_{2j}\cos(\varphi_{j2}) & \dots & B_{2N}\cos(\varphi_{N2})
  \\
    \vdots & \vdots & \ddots & \vdots
  \\
    B_{N1}\cos(\varphi_{1N}) &  B_{N2}\cos(\varphi_{2N}) & \dots & - \sum_{j\neq N} B_{Nj}\cos(\varphi_{jN})
  \end{pmatrix}
\ ,
\end{equation}
where $\varphi_{ji} := \theta^*_j({\omega(T)}) - \theta^*_i({\omega(T)})$ is independent of ${\omega(T)} \in \RR$ and the phase-locked solution $\theta^*$. 
Note that $A$ has row sum zero and therefore $0$ is an eigenvalue.
The following theorem characterizes the fact that the manifold $\cal M$ is stable with asymptotic phase by properties of \eqref{KuramotoTransformed} on and in the vicinity of $\cal N$. For an alternative approach using set-valued Lyapunov functions see e.g.\ \cite{Goebel2010} and the references therein.

\begin{thm}\label{thm1} Assume that $\theta^*$ is a phase-locked solution of \eqref{Kuramoto} and $\operatorname{rank} A = N-1$. Then the following three statements are equivalent:
\begin{itemize}
  \item[(i)] $\theta^*$ is \emph{orbitally stable} in the sense that $\cal M$ is a stable manifold of phase-locked solutions of \eqref{Kuramoto} with asymptotic phase.
  \item[(ii)] $\cal N$ is a stable center manifold of equilibria of \eqref{KuramotoTransformed} with asymptotic phase, i.e.\ there exists a neighborhood $V$ of $\cal N$ such that solutions starting in $V$ converge for $t \to \infty$ to an equilibrium in $\cal N$. 
  \item[(iii)] The manifold $\cal N$ of equilibria of \eqref{KuramotoTransformed} satisfies the following two conditions:
  \begin{itemize} 
    \item[(a)] linear stability: $A$ has no eigenvalues with positive real part, i.e.\ $A$ is negative semi-definite.
    \item[(b)] global attractivity: there exists a  neighborhood $V$ of $\cal N$, such that solutions starting in $V$ converge for $t \to \infty$ to $\cal N$. 
  \end{itemize}
\end{itemize}
\end{thm}

\begin{proof}
$(iii) \Rightarrow (ii)$. This follows from \cite[Proposition 4.1 \& proof of Theorem 4.1]{Aulbach1984}.

$(ii) \Rightarrow (iii)$. (iii)(b) follows directly from (ii). To prove (iii)(a), assume that $A$ has an eigenvalue with positive real part. Then the linearization $\dot x = Ax$ has an unstable subspace giving rise to unstable manifolds for every equilibrium $x^{({\omega(T)})}$ of \eqref{KuramotoTransformed}, which contradicts the fact that every solution in $V$ converges to an equilibrium in $\cal N$ for $t \to \infty$.

$(i) \Leftrightarrow (ii)$. The transformation $x_i(t) = \theta_i(t) - \Omega t$ maps $\cal M$ to $\cal N$ and preserves the property of asymptotic phase of corresponding solutions.
\end{proof}

Our main aim for the Kuramoto model in this paper is to
\begin{itemize}
  \item[(i)] reveal the dependence of the stability of the continuum of equilibria on certain topological properties of the coupling network,
  \item[(ii)] efficiently evaluate algebraic stability criteria.
\end{itemize}
To achieve this goal, we discuss a characterization of condition (iii)(a) of Theorem \ref{thm1} which shows that semi-definiteness is a necessary condition for stability.

Semi-definiteness can be characterized in terms of the determinants of the submatrices (minors) of a matrix via Sylvester's criterion. Stability analysis by means of Sylvester's criterion is well-known in control theory and has been applied to problems of different fields from fluid- and thermodynamics to offshore engineering and social networks (see \cite{do-gross2012} and the references therein).
From a modern perspective the application of Sylvester's criterion is very appealing as it can reveal stability criteria on different scales.  
The smallest minor considered by Sylvester's criterion is just a diagonal element of the Jacobian matrix and thus poses a condition on the internal dynamics on a single vertex. The subsequent minors pose conditions on subgraphs such as vertex pairs, triplets, and so on. Sylvester's criterion can thereby contribute to the search of meso-scale structural properties that have an impact on the global dynamics of networks, which is becoming a current topic in physics \cite{ChaosFocusIssue,Frey2012}. 
However, due to the rapidly growing complexity of the minors to be evaluated, the application of the criterion was previously limited to systems with few degrees of freedom. We present a way to apply Sylvester's criterion to systems with many degrees of freedom. Generalizing the classical Kirchhoff Theorem, we obtain a characterization of the minors by properties of forests in the coupling graph of our system. 

We recall some basic notions from graph theory.
A weighted graph $G$ is a triple $(V,E,\omega)$ consisting of
a finite non-empty set of vertices $V$, a set of edges $E \subseteq V^{(2)}$ consisting of unordered pairs 
of vertices and a weight function $\omega\colon E \rightarrow \RR \setminus \{0\}$.

A \emph{walk} of length $\ell$ from $i_0 \in V$ to $i_\ell \in V$ is a sequence of vertices $i_0,\ldots,i_\ell$ such that 
$\{i_{k},i_{k+1}\} \in E$ for $k \in \{1,\ldots,\ell-1\}$. A \emph{cycle} is a walk
for which $i_0=i_\ell$ and $i_k \neq i_{k'}$ for $0<k<k'$. 
A graph is \emph{connected} if for each pair of vertices $(i,j)$ there is a walk from $i$ to $j$.
A \emph{forest} is a graph without cycles, a \emph{tree} is a connected forest.

A graph $G'=(V',E',\omega')$ is a subgraph of $G=(V,E,\omega)$ if $V' \subseteq V$, $E' \subseteq E$ and $\omega'=\omega_{|E'}$.
If $V'=V$ the subgraph is \emph{spanning}. A \emph{spanning tree} of a graph 
is a subgraph that is spanning and a tree.

We will identify subsets $E' \subseteq E$ with the \emph{induced subgraph} $G' = (V',E',\omega')$ of $G$, where $V' = \{i \in V \colon \exists j \in V \textup{ with } \{i,j\} \in E'\}$ and $\omega' = \omega_{\mid E'}$.

A \emph{cut} of the graph $G$ given by a partition $V_1,V_2$ of the vertex set is the set of edges $E(V_1,V_2)$ between $V_1$ and $V_2$.

The weighted adjacency matrix $\AC(G) \in \RR^{n \times n}$ of a graph $G=(V,E)$ with $V=\{1,\ldots,n\}$ is defined by 
\[
  \AC(G)_{ij}=\begin{cases}
    \omega(i,j) & \text{if } \{i,j\} \in E\\
    0 & \text{otherwise}
  \end{cases}.
\]

For a symmetric matrix $A = (a_{ij}) \in \RR^{n \times n}$ the \emph{Coates-graph}
$\GC(A)=(V,E,\omega)$ is defined as the weighted undirected graph with vertices $V=\{1,\ldots,n\}$, edges 
$E =\{ \{i,j\} \colon a_{ij} \not=0\}$ and weight function $\omega(\{i,j\}):=a_{ij}$.
On the other hand, for each undirected graph its associated adjacency matrix is symmetric.

The \emph{Laplace-matrix} $\LC(G)$ of a graph $G$ is defined by $\LC(G)=\DC(G)-\AC(G)$ where $\DC(G)$ is the diagonal \emph{degree matrix} $\DC(G):=(\AC(G)\ones)_i$. Take for example $G$ with
\begin{flalign*}
 \AC(G)=\begin{pmatrix}
  0  &  \frac{1}{2} &  0 & -3 \\
 \frac{1}{2} &0 & 1 & -2 \\
 0 & 1 & 0 & 1 \\
 -3 & -2 & 1 & 0
 \end{pmatrix},
\textup{then} \qquad
&
\LC(G)=\begin{pmatrix}
  -2\frac{1}{2}  &  -\frac{1}{2} &  0 & 3 \\
 -\frac{1}{2} &-\frac{1}{2} & -1 & 2 \\
 0 & -1 & 2 & -1 \\
 3 & 2 & -1 & -4
 \end{pmatrix}.
\end{flalign*}

\begin{rem}\label{rem-adjlap}
(a) If loops in $G$ are removed, then $\LC(G)$ does not change.

(b) If $G$ has an adjacency matrix $\AC(G)$ which has row sum zero, then $\LC(G)=-\AC(G)$. 
\end{rem}

The structure of the paper is as follows. In Section 2 we recall Sylvester's criterion in Theorem \ref{thm-sylvester}. Remark \ref{rem-checkdef} emphasizes with an example the fact that testing for semi-definiteness is fundamentally more costly than checking definiteness. The main result of this section, Theorem \ref{thm-innerminors}, is a version of Sylvester's criterion for zero-row-sum matrices of maximal rank. Section 3 is devoted to the Matrix Tree Theorem \ref{thm-matrix-tree} for principal minors. Its proof is split in three parts (Lemmas \ref{lem-cb-minors}, \ref{lem-Laplacian-Incidence} and \ref{lem-charact-MSL}). Section 4 contains Lemma \ref{lem-cutting} on how to cut a forest, which is used to prove the combinatorial identity in Theorem \ref{thm-identity}. Section 5 finally proves the existence of a positive spanning tree in the connected components of Coates graphs of negative semi-definite matrices and gives some applications of this theorem.

\section{Definiteness and Sylvester's criterion}
In this section we recall Sylvester's criterion and adapt it to the case of positive semi-definite matrices with zero row sum and maximal rank.

Let $p, q \in \NN, S \subseteq \{1,\ldots,p\}, T \subseteq \{1,\ldots,q\}$ with $|S|=|T|$ and $B \in \RR^{p \times q}$.
We denote the submatrix of $B$ defined by $S$ and $T$ by $B_{S,T}$.
The determinant of this submatrix is called the \emph{minor} corresponding to $S$ and $T$
for which we write $[B]_{S,T}$.

If $B$ is a square matrix, we call $B_{S,S}$ a \emph{principal submatrix} and
its determinant a \emph{principal minor}. If additionally
$S=\{1,\ldots,k\}$ for some $k \in \NN$, we call it a \emph{leading principal submatrix/minor}.

Recall that  $L\in \RR^{n \times n}$ is called
\begin{itemize}
 \item \emph{positive (semi) definite} iff $\forall v \in \RR^n\colon v^T L v > 0 \:\:(v^T L v \geq 0)$,
 \item \emph{negative (semi) definite} iff $\forall v \in \RR^n\colon v^T L v < 0 \:\:(v^T L v \leq 0)$.
\end{itemize}

\begin{thm}[Sylvester's criterion, {\cite[Theorem 7.2]{Zhang2011}}]\label{thm-sylvester}
A symmetric matrix $L \in \RR^{n \times n}$ is positive definite iff every leading principal minor of $L$ is positive. $L$ is positive semi-definite iff every principal minor of $L$ is non-negative, i.e.\ $[L]_{S,S} \geq 0$ for all $S \subseteq \{1,\ldots,n\}$.
\end{thm}

\begin{rem}[Checking for semi-definiteness]\label{rem-checkdef}
Note that for definiteness one has to check the sign of only $n$ minors of $L$,
 while one has to check in general the sign of all $2^n$ principal minors to test for semi-definiteness.
 Even if we know that $0$ is a simple eigenvalue of $L$, $\GC(L)$ is connected and $L$ has eigenvector $\ones$,
it is not enough to check the leading principal minors for non-negativity to ensure semi-definiteness. Consider the following matrix:
\vspace{-0.5cm}
\begin{flalign*}
 C&=\begin{pmatrix}
    0  &  0 &  1 & -1 \\
    0  & -1 &  1 & 0 \\
    1  &  1 &  -2 & 0 \\
   -1  &  0 &  0 & 1
    \end{pmatrix}&
\GC(C)&=\begin{tikzpicture}[baseline=0.9cm]
\tikzstyle{LabelStyle}=[fill=white]
           \Vertex[x=0,y=0] {A}
           \Vertex[x=2,y=0] {B}
           \Vertex[x=2,y=2] {C}
           \Vertex[x=0,y=2] {D}
           \Edge[label=$1$](A)(C)
           \Edge[label=$-1$](A)(D)
           \Loop[label=$-1$,dir=EA,dist=1cm](B)
           \Edge[label=$1$](B)(C)
           \Loop[label=$-2$,dir=EA,dist=1cm](C)
           \Loop[label=$1$,dir=EA,dist=1cm](D)
          \end{tikzpicture}.
\end{flalign*}
The corresponding graph is connected, the corresponding characteristic polynomial is $x^4+2x^3-4x^2-4x$ (hence
zero is a simple eigenvalue), the leading principal minors are
\begin{align*}
 [C]_{\{1,2,3\},\{1,2,3\}}&=\begin{vmatrix}
    0  &  0 &  1 \\
    0  & -1 &  1 \\
    1  &  1 & -2 \\  
    \end{vmatrix}=1\\
[C]_{\{1,2\},\{1,2\}}&=\begin{vmatrix}
    0  &  0  \\
    0  & -1  \\
    \end{vmatrix}=0\\
[C]_{\{1\},\{1\}}&=\begin{vmatrix}
    0  
    \end{vmatrix}=0
.
\end{align*}
But $C$ is not positive semi-definite since for $v=(1,1,0,0)^T$ we have $v^T C v=-1 <0$.
\end{rem}

The following theorem is a version of Sylvester's criterion for zero-row-sum matrices.

\begin{thm}[Sylvester's criterion for zero-row-sum matrices of maximal rank]\label{thm-innerminors}
Let $L \in \RR^{n \times n}$ be a symmetric matrix with zero row sum.
Then the following are equivalent:
\begin{enumerate}[(i)]
\item $L$ is positive semi-definite and has rank $n-1$
\item $[L]_{S,S}>0$ for all $S \varsubsetneq \{1,\ldots,n\}$
\item $L_{S,S}$ is positive definite for all $S \subsetneq \{1,\ldots,n\}$
\item $[L]_{S,S}>0$ for all $S=\{1,\ldots,k\},\: k \in \{1,\ldots,n-1\}$
\item $L_{S,S}$ is positive definite for $S =\{1,\ldots,n-1\}$
\end{enumerate}
\end{thm}
\begin{proof}
The equivalences $(ii) \Leftrightarrow (iii)$ and $(iv) \Leftrightarrow (v)$ follow directly from Sylvester's criterion.
Since $(iii)$ contains $(iv)$ as a special case, there remain two implications to show.

  $(i) \Rightarrow (ii)$
    From Sylvester's criterion we know that $[L]_{S,S}\geq 0$ for all $S \subseteq \{1,\ldots,n\}, 0<|S|<n$.
    We will show that having a proper principal minor equal to zero will imply
    that zero is not a simple eigenvalue of $L$ under the given assumptions.
    By eigenvalue interlacing \cite{GodsilRoyle2001} we know that for each symmetric matrix $B \in \RR^{n \times n}$ 
    and each $S \subseteq \{1,\ldots,n\}, 0<|S|\leq n$
    the following inequality holds:
    \begin{align*}
     \lambda_{\text{min}} (B) \leq \lambda_{\text{min}} (B_{S,S})
    \end{align*}
    where $\lambda_{\text{min}}$ denotes the minimal eigenvalue.
    Now assume that $[L]_{S,S}=0$ for some $S \subseteq \{1,\ldots,n\}, 0<|S|<n$.
    Then there is $i \in \{1,\ldots,n\} \setminus S$. Define $\tilde{S}=\{1,\ldots,n\} \setminus i$.
    Then as seen above $0 =\lambda_{\text{min}}(L_{S,S}) \geq \lambda_{\text{min}}(L_{\tilde{S},\tilde{S}}) \geq 0$ and thus $[L]_{\tilde{S},\tilde{S}}=0$
    and there is an eigenvector $\tilde{v}$ of $L_{\tilde{S},\tilde{S}}$ for the eigenvalue zero.
    We will now show how to get an eigenvector of $L$ for the eigenvalue zero independent of $\ones$ thus
    giving a contradiction to zero being a simple eigenvalue of $L$.
    
    Assume wlog that $i=n$. Then, because of the zero row sum assumption, $L$ can be decomposed as
\begin{align*}
 \left(\begin{array}{ccc|c}
    ~&~&~&\\
    &\tilde{L} & &-\tilde{L}\ones \\
    ~&~&~& \\ \hline
    &-\ones^T \tilde{L} && \ones^T \tilde{L} \ones
 \end{array}\right).
\end{align*}
But then
\begin{align*}
 \left(\begin{array}{ccc|c}
    ~&~&~&\\
    &\tilde{L} & &-\tilde{L}\ones \\
    ~&~&~&\\ \hline
    &-\ones^T \tilde{L} && \ones^T \tilde{L} \ones
 \end{array}\right)
\left(\begin{array}{c}
       \tilde{v}\\ 0
      \end{array}
\right)
&=\left(\begin{array}{c}
    ~\\
    \tilde{L}\tilde{v} + 0 \\
    ~\\
    \hline
    -\ones^T \tilde{L}\tilde{v} +0
 \end{array}\right)
=0.
\end{align*}

$(v) \Rightarrow (i)$ If $L$ is not positive semi-definite or has rank smaller than $n-1$, there is an eigenvalue $\lambda\leq 0$ of
$L$ and a corresponding eigenvector $v=(v_1,\ldots,v_n)^T$ independent of $\ones$. Define $u=v - v_n\ones$ and $\tilde{u}=(u_1,\ldots,u_{n-1})\neq 0$.
Since $u_n=0$ we have 
\begin{align*}
  \tilde{u}^T L_{\{1,\ldots,n-1\},\{1,\ldots,n-1\}} \tilde{u} 
  & = u^T L u \\
  & = v^TLv -2v_n \ones^T L v+v_n^2 \ones^T L \ones \\
  & = v^TLv=\lambda v^Tv\leq 0.
\end{align*}
But this contradicts $L_{\{1,\ldots,n-1\},\{1,\ldots,n-1\}}$ being positive definite.
\end{proof}

\section{The Principal Minor Matrix Tree Theorem}

We are looking for properties of a symmetric, zero row sum $n \times n$ matrix $A \subseteq \RR^{n\times n}$ that are necessary for negative semidefiniteness. These properties are best described in terms of the structure of $A$, more precisely, the position and weight of the nonzero entries of $A$. Graph theory provides suitable concepts such as subgraphs, trees and cycles to describe this structure. In the light of remark \ref{rem-adjlap} $A$ is negative semidefinite
iff $\LC(G)=-A$ is positive semidefinite where $G=\GC(A)$ is the Coates graph of $A$. In this section, we therefore seek a combinatorial interpretation of the minors of $\LC(G)$  in terms of the subgraphs of $G$. 
For $S \subseteq \{1,\ldots,n\}, S\neq \emptyset,$ we define 
\begin{align*}
\FC_S:=\big\{K \subseteq E(G) \:\big|\: &|K|=|S|, \text{each connected component in } K \text{ contains at least one }\\
&\text{vertex not in } S\}.
\end{align*}

The definition of $\FC_S$ implies two features which are characterized by the next lemma.
\begin{lem}\label{lem-fcs-forest} Let $K \in \FC_S$. Then the following holds:\\
(a) $K$ is a forest.\\
(b) each tree in $K$ contains exactly one vertex not in $S$. 
\end{lem}
\begin{proof}
(a) Assume $K$ is not a forest, i.e.\ it contains a cycle.
    Each connected component $C$ of $K$ has at most $|E(C)|$ vertices in $S$ and
    the component $C'$ containing the cycle contains at most $|E(C')|-1$ vertices in $S$. This gives
        \begin{align*}
          |S| \leq \sum_{C \text{ is conn.\ comp.\ of } K} |V(C) \cap S | \leq \Big(\sum_{C} |E(C)|\Big) -1 = |K|-1
              = |S|-1
        \end{align*}
        which is a contradiction.
\\[0.5ex]
(b) Assume there is a connected component $C'$ of $K$ with more then one vertex outside $S$.
    Then $C'$ contains at most $|E(C')|-1$ vertices in $S$ and we get the same contradiction as in (a).
\end{proof}

\begin{figure}
\includegraphics[width=\textwidth]{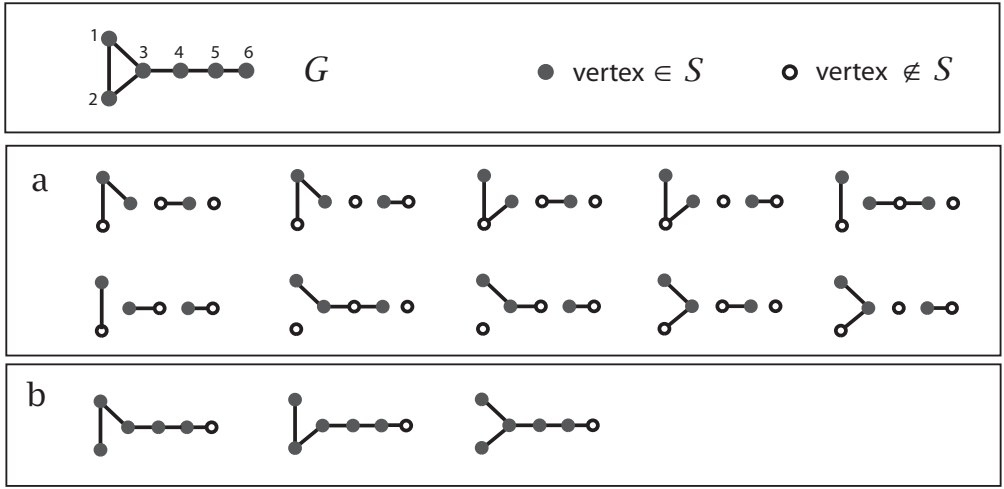}
\caption{Examples for $\FC_S$. Shown are the sets $\FC_S$ of a graph $G$ (upper panel) for $S=\{ 1,3,5\}$ (a),  $S=\{ 1,2,3,4,5\}$ respectively (b). Panel (b) illustrates the relation between Theorem~\ref{thm-matrix-tree} and Kirchhoff's Theorem: For all $G$ and all sets $S$ with $|S|=n-1$, $\FC_S$ is the sum over all spanning trees of $G$.  \label{Fs_fig}}
\end{figure}


\begin{thm}[Matrix Tree Theorem for principal minors and undirected graphs]
\label{thm-matrix-tree}
Let $G$ be an undirected weighted graph, $S \subseteq V(G)$. Then
\begin{align*}
 [\LC(G)]_{S,S} &= \sum_{K \in \FC_S} \omega(K)\\
\text{where } \omega(K)&=\prod_{e \in K} \omega(e).
\end{align*}
\end{thm}

This theorem can be proved using combinatorial arguments as in \cite{Chaiken}, but the well known 
proof of the classical matrix tree theorem using the Cauchy-Binet theorem can easily be adapted to
this generalization. To keep the paper self-contained we include the proof. Its structure 
is as follows. The Laplacian of the graph is expressed as the product
of its incidence matrix, its transpose and a weight matrix (Lemma \ref{lem-Laplacian-Incidence}). By the Cauchy-Binet theorem
the minors of this product are expressed as a sum over products of incidence
matrix minors (Lemma \ref{lem-cb-minors}). Finally these minors are interpreted in terms of $\FC_S$ (Lemma \ref{lem-charact-MSL}).

\begin{lem}[Cauchy-Binet-Theorem for minors \cite{Zhang2011}]\label{lem-cb-minors}
 For an $n \times m$ matrix $D$, an $m \times n$ matrix $E$ and index sets $I,J \subseteq \{0,\ldots,n\}$ with $|I|=|J|\leq m$ we have
 \begin{align*}
  [DE]_{I,J} = \sum_{\substack{K \subseteq \{1, \ldots, m\}, |K|=|I|}} [D]_{I,K} [E]_{K,J}
 \end{align*}
\end{lem}

For the proof of Theorem  \ref{thm-matrix-tree} we first have to introduce some notation:
First we impose an arbitrary orientation on $G$. Note that the proofs of \ref{lem-Laplacian-Incidence} and \ref{lem-charact-MSL} are independent of this orientation since the entries $M_{ie}$ of $M$ occur only in expressions of the form $M_{ie} M_{je}$ or $M_{ie}+M_{je}$ for some $e \in \{1,\ldots,|E|\}$.

The incidence-matrix $M \in \{-1,0,1\}^{n \times |E(G)|}$ of $G$ is then defined by
\begin{align*}
 M_{ie}=\begin{cases}
         1 & \text{ if } e \text{ starts at } i \\
         -1 & \text{ if } e \text{ ends at } i \\
         0 & \text{ if } i \not\in e\\
        \end{cases}.
\end{align*}
Additionally, we define the weight matrix $W \in \RR^{|E(G)|\times |E(G)|}$ as the diagonal matrix
with $W_{ee}=\omega(e)$.

The importance of these incidence matrices for our application arises from the following two lemmas.

\begin{lem}[Representation of $\LC(G)$ by the incidence and weight matrix]\label{lem-Laplacian-Incidence}
 \begin{align*}
  MWM^T = \LC(G)
 \end{align*}
\end{lem}
\begin{proof} 
For the off-diagonal entries $(MWM^T)_{ij}$ with $i \neq j$ we have
 \begin{align*}
  (MWM^T)_{ij}=&\sum_{e \in E(G)} \underbrace{M_{ie} W_{ee} M_{je}}_{=:*}\\
*&=\begin{cases}
                                                          0 & \text{ if } e \neq \{i,j\} \\
                                                          -\omega(e) & \text{ if } e = \{i,j\}
                                                          \end{cases},\\
 \end{align*}
which equals $\LC(G)_{ij}$. The same holds for the diagonal entries
 \begin{align*}
             (MWM^T)_{ii}=&\sum_{e \in E(G)} M_{ie}^2 W_{ee} = \sum_{\substack{e \in E(G) \\ i \in e}} \omega(e)=\LC(G)_{ii}. \qedhere
 \end{align*}
\end{proof}

\begin{lem}[Characterization of the incidence matrix minors via $\FC_S$]
\label{lem-charact-MSL}
 \begin{align*}
  | [M]_{S,K} | = \begin{cases} 1 &\text{ if } K \in \FC_S \\ 0 &\text{ if } K \not\in \FC_S \end{cases} \quad\text{ for all } S \subseteq V, K \subseteq E, |K|=|S|
 \end{align*}
\end{lem}
\begin{proof}
($K \not\in \FC_S \Rightarrow [M]_{S,K}=0$)
 If $K \not \in \FC_S$, there is a connected component $C$ of $K$ that has no vertex outside $S$. Then the rows of $M_{S,K}$ corresponding
         to these vertices, sum up to $0$.

($K \in \FC_S \Rightarrow [M]_{S,K}=1$)
      We prove this direction by induction on $|S|$. If $|S|=1$ and $K \in \FC_S$ then $K$ is a single edge between a vertex
      in $S$ and one outside $S$, so $|[M]_{S,K}|=1$. If $|S| > 1$ then by Lemma \ref{lem-fcs-forest} there is a tree in $K$, which
      has at least one leaf $i$ in $S$ and an edge $e \in K$ incident with $i$ which implies that $K \setminus \{e\} \in \FC_{S \setminus \{i\}}$. Since the row of $M_{S,K}$ corresponding to $i$ contains only one nonzero element,
      by expansion of the determinant along this row 
      $|[M]_{S,K}|=|[M]_{S \setminus \{i\},K \setminus \{e\}}|$ which implies the assertion by induction. \qedhere

\end{proof}

\begin{figure}
\includegraphics[width=\textwidth]{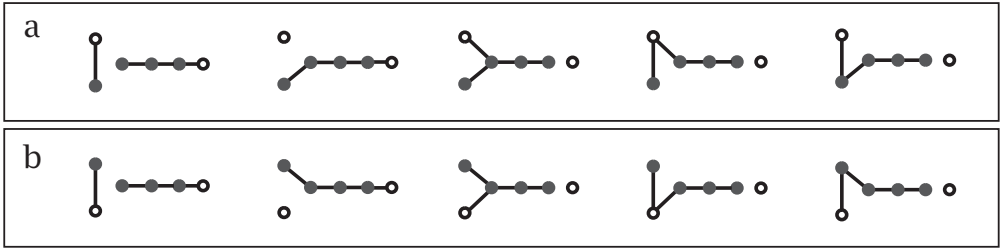}
\caption{Illustration of the induction step used in the proof of Lemma~\ref{lem-charact-MSL}. 
Consider the graph $G$ from Fig.~\ref{Fs_fig}. As shown in Fig.~\ref{Fs_fig}(b), every tree in $\FC_S$, $S=\{1, 2,3,4,5\}$, has at least one leaf $i$ in $S$ and an edge $e \in L$ incident with $i$. Comparison with the sets $\FC_{S \setminus \{1\}}$, $\FC_{S \setminus \{2\}}$, displayed in (a), (b) respectively, reveals that for all trees and all $i$, $K \setminus \{e\}\in\FC_{S \setminus \{i\}}$. \label{Induction_fig}}
\end{figure}
We are now ready to proof Theorem \ref{thm-matrix-tree}.

\begin{proof}[Proof of Theorem \ref{thm-matrix-tree}] 
Using Lemma 9, Lemma 8 and Lemma 10, we get
\begin{align*}
 [\LC(G)]_{S,S} = [MWM^T]_{S,S} &= \sum_{\substack{K \subseteq E \\ |K|=|S|}} [MW]_{S,K} [M^T]_{K,S} \\
                               &= \sum_{K} [M]_{S,K}^2 \prod_{e \in K} \omega(e) \\
                               &= \sum_{K \in \FC_S}  \prod_{e \in K} \omega(e) \\
                               &= \sum_{K \in \FC_S}  \omega(K). \qedhere \\
\end{align*}
\end{proof}

\section{Cutting a Forest}
We now have a convenient characterization of the minors of a zero-row-sum symmetric matrix $L \in \RR^{n \times n}$ in terms of forests contained in $G=\GC(-L)$. 
To show that a given structure of $G$ is incompatible with the positive semi-definiteness of $L$, one needs to show that the inequalities $[L]_{S,S}>0$, $S \varsubsetneq \{1,\ldots,n\}$, cannot simultaneously be fulfilled.

However, if one can show that e.g.\ $\sum_{S} [L]_{S,S}\leq 0$, one may dispense with the explicit calculation of any of the inequalities as well as with the compatibility verification.
In this section, we prove a combinatorial identity that is central for this type of calculation. 

Let $V_1,V_2$  be a partition of the vertex set of $G$ into two non-empty parts.
For $B\subseteq V_1$ define
\begin{align*}
  \Sigma_B
  &=
  \big\{
  D \subseteq E \cap E(V_1, V_2) \;\big|\; D \text{ is a forest and}
\\ 
  & \phantom{= \big\{D \subseteq E \cap E(V_1,V_2) \setsep | \,}
  \text{contains exactly one edge } e \text{ for each } i \in B \text{ with } i \in e\big\} 
  ,
\\
  T_B
  &=
  \big\{D \subseteq E \cap E(V_1,V_1) \;\big|\; |D|=|V_1|-|B|, D \text{ is a forest and }
\\
  & \phantom{=\{D \subseteq E \cap E(V_1,V_1) \setsep |} 
  \text{ each tree contains exactly one vertex of } B
  \big\}
  . 
\end{align*}

\begin{lem}[Cutting a forest]\label{lem-cutting}  
Let $C \subseteq V_1$. Then
 \begin{align*}
 \FC_{V_1 \setminus C} = \bigcup_{\substack{B \subseteq V_1 \setminus C \\
                                            \Sigma_B \neq \emptyset \\
                                            T_{C \cup B} \neq \emptyset \\
                                            }} \{A \cup A' \setsep A \in \Sigma_B, A' \in T_{C \cup B}\}.
\end{align*}
\end{lem}

\begin{proof}
$(\subseteq)$. For $K \in \FC_{V_1 \setminus C}$ we clearly have $K \cap E(V_1,V_2) \in \Sigma_B$ where $B$
is the set of all indices of $V_1$ incident to an edge of $K$ between $V_1$ and $V_2$.
By definition of $\FC_{V_1 \setminus C}$ we also have $K \cap E(V_2,V_2) = \emptyset$ and 
therefore $|K \cap E(V_1,V_1)|=|V_1 \setminus C|-|B|$. This gives $K \cap E(V_1,V_1) \in
T_{C \cup B}$.

$(\supseteq)$. On the other hand let $K_S \in \Sigma_B$ and $K_T \in T_{B \cup C}$. Then 
each tree in $K:=K_S \cup K_T$ contains exactly one vertex not in $V_1 \setminus C$ and
$|K|=|V_1| - |C \cup B|+ |B|=|V_1 \setminus C|$. Therefore $K \in \FC_{V_1 \setminus C}$.\qedhere
\end{proof}

This decomposition of $[\LC(G)]_{V_1 \setminus C,V_1 \setminus C}$ allows us to prove the following remarkable identity.

\begin{thm}\label{thm-identity}
 \begin{align*}
  \sum_{C \subseteq V_1} (-1)^{|C|} {\omega(\Sigma_{C})} [\LC(G)]_{V_1 \setminus C,V_1 \setminus C} = 0
 \end{align*}
\end{thm}
\begin{proof} 
From Lemma \ref{lem-cutting} we can derive 
\begin{align*}
 [\LC(G)]_{V_1 \setminus C,V_1 \setminus C}&=\sum_{K \in \FC_{V_1 \setminus C}} \omega(K) \\
             &=\sum_{B \subseteq V_1 \setminus C} \sum_{\substack{K_S \in \Sigma_B \\ K_T \in T_{B \cup C}}} \omega(K_S) \omega(K_T)\\
             &=\sum_{B \subseteq V_1 \setminus C} \left(\sum_{K_S \in \Sigma_B} \omega(K_S)\right)\left(\sum_{K_T \in T_{B \cup C}} \omega(K_T)\right)\\
             &=\sum_{B \subseteq V_1 \setminus C} {\omega(\Sigma_B)} {\omega(T_{B \cup C})},\\
  \sum_{C \subseteq V_1} (-1)^{|C|} {\omega(\Sigma_{C})} [\LC(G)]_{V_1 \setminus C,V_1 \setminus C} &= \sum_{C \subseteq V_1} (-1)^{|C|} {\omega(\Sigma_{C})} \sum_{B \subseteq V_1 \setminus C} {\omega(\Sigma_B)} {\omega(T_{B \cup C})} \\
&=  \sum_{\substack{C \subseteq V_1\\ B \subseteq V_1 \setminus C}} (-1)^{|C|} {\omega(\Sigma_{B \cup C})} {\omega(T_{B \cup C})} \\
&=  \sum_{\substack{A \subseteq V_1\\ C \subseteq A}} (-1)^{|C|} {\omega(\Sigma_{A})} {\omega(T_{A})} \\
&=  \sum_{A \subseteq V_1}  {\omega(\Sigma_{A})} {\omega(T_{A})} \sum_{C \subseteq A} (-1)^{|C|}=0, \\
\end{align*}
where we use the fact that \[\sum_{C \subseteq A} (-1)^{|C|} = \sum_{k=1}^{|A|} \binom{|A|}{k} (-1)^{k} 1^{|A|-k} =(-1+1)^{|A|}=0.\qedhere\] 
\end{proof}

\section{Positive Spanning Trees}

A weighted graph is said to be \emph{positive} or \emph{negative}, if each of its edge weights is positive or negative, resp.

\begin{lem}
 Let $G$ be a weighted connected graph. Then $G$ has a positive spanning tree iff it has no 
negative cut.
\end{lem}
\begin{proof}
 Take an arbitrary cut of $G$ given by the partition $V_1, V_2$ of $V$. If $G$ has a positive spanning
 tree, then there must be an edge from $V_1$ to $V_2$ with positive weight, so the cut is not negative.
 If on the other hand there is no negative cut, remove all negative edges from $G$. Then
 each cut still contains an edge and therefore the remaining graph contains a spanning tree with
 all edge weights positive.
\end{proof}

From now on let $A$ be a real symmetric negative semi-definite matrix with zero row sum.
\begin{thm}\label{thm:Pos_spanningtree}
If $G=\GC(A)$ is connected then
$\GC(A)$ has a positive spanning tree.
\end{thm}
\begin{proof}
 Assume there is no positive spanning tree. By the above theorem this implies the existence of a negative cut given
 by the partition $V_1,V_2$ of $V(\GC(A))$, that is $\forall e \in V_1 \times V_2 \cap E:\: \omega(e)<0$.
 For $i \in V(\GC(A))$ define $z_i= \sum_{j \in V_1} A_{ij}$ and $z_i'=\sum_{j \in V_2} A_{ij}$.
 Since $A$ has zero row sum, $z_i+z'_i=0$ und by the negativity of the above cut
$\forall i \in V_1:\: z'_i \leq 0, \exists i \in V_1: \: z'_i<0$. Finally consider $v \in \RR^n$ with
\[v_i:=\begin{cases}
      1 & \text{ if } i \in V_1 \\
      0 & \text{ if } i \in V_2 \\
     \end{cases}.\]
Then $v^T A v= \sum_{i,j \in V_1 \times V_1} A_{ij} = \sum_{i \in V_1} z_i >0$, thus $A$ is not negative
semi-definite, a contradiction.
\end{proof}

\begin{rem} For matrices of rank $n-1$, one can also derive this result from Theorem \ref{thm-identity}. Since $V_1,V_2$ define a negative cut, $(-1)^{|C|} \omega(\Sigma_C)$ is non-negative for all $C$ and positive iff all vertices in $C$ are incident to an edge between
$V_1$ and $V_2$. Since according to Theorem \ref{thm-innerminors} all minors $\LC(G)_{S,S}$ with $|S| <n$ are positive
if our system is stable, the left-hand-side of the identity in Theorem \ref{thm-identity} is positive,
whereas the right-hand-side is zero.
\end{rem}

\begin{rem}
If $\GC(A)$ has disconnected components, every such component has a positive spanning tree. 
\end{rem}

We now apply Theorem \ref{thm:Pos_spanningtree} and Theorem \ref{thm-matrix-tree} to derive an upper bound on the absolut weight 
on negative edges in induced lines.

\begin{thm}\label{thm-weightbound}
Let $H$ be an \emph{induced line} of $\GC(A)=(V,E,\omega)$ with at least two edges, i.e.\ a connected induced subgraph  for which all but  two vertices
have degree two in $\GC(A)$.
Then $H$ contains at most one negative edge $e$ and if it contains one, 
the absolute value of its weight is bounded from above by 
\begin{align*} \frac{\prod_{a \in E(H)\setminus \{e\}}\omega(a)}{\sum_{b \in E(H)\setminus \{e\}} \prod_{a \in E(H)\setminus \{e,b\}} \omega(a)}=\frac{1}{\sum_{a \in E(H) \setminus \{e\}} \frac{1}{\omega(a)}}.\end{align*}
\end{thm}
\begin{proof}
Let $T$ be a spanning tree of $\GC(A)$. Then $|E(H)\setminus E(T)|\leq 1$, hence $H$ contains at most one negative edge
by Theorem \ref{thm:Pos_spanningtree}.
By definition of an induced line, $H$ is a tree with two leaves. Let $\{i_1,\ldots, i_{|H|}\} = V(H)$
be an enumeration of the vertices of $H$ such that $E(H)=\{(i_k,i_{k+1}) \setsep k \in \{1,\ldots,|H|-1\}\}$.
The bound can immediately be derived from Theorem \ref{thm-matrix-tree} applied to 
$S=\{i_2,\ldots, i_{|H|-1}\}$. Notice that the elements of $\FC_S$ are
precisely the subsets of $E(H)$ of
size $|H|-2$.
The bound on the weight of a negative edge is then derived by using the non-negativity of the principal minor $[A]_{S,S}$.
\end{proof}

Every connected subgraph of an induced line is again an induced line. Hence for each 
induced line in which a negative edge is contained, Theorem \ref{thm-weightbound}
gives a bound on the weight of this edge. If one such line is
contained in another, the longer one yields a sharper bound. 

\begin{exm} We now apply the results of this section to the Kuramoto model defined in Eq.\ \eqref{Kuramoto} and reveal common properties of the 
configuration of all possible phase-locked systems. Consider Theorem \ref{thm:Pos_spanningtree} with the edge weights graph $G=\GC(A)$ given 
by $A_{ij}=B_{ij}\cos(\varphi_{ji})$. Since all coupling constants $B_{ij}\geq 0$, stability of the phase-locked state requires that each component of the physical 
coupling network has a spanning tree of oscillators obeying \mbox{$|\varphi_{ji} |<\frac{\pi}{2}$}. 

Theorem \ref{thm-weightbound} implies that
stability of the phase-locked state can only occur if along each induced line $H$ of $G$ there are at most two neighbouring oscillators with 
absolute phase-difference greater or equal $\frac{\pi}{2}$. If such a pair $(i,j)$ of oscillators exists, the phase-locked state
can only be stable if the absolut value of the reciprocal of $\cos(\varphi_{ji})$ is larger then the coupling constant $B_{ij}$ times
the sum of the reciprocals of all other coupling constants $B_{k\ell}$ in $H$.
\end{exm}

\section{Conclusion}

In the present paper we used an extension of Sylvester's criterion to formulate meso-scale obstructions to stability of one-dimensional center manifolds, in particular, of the center manifold of phase-locked solutions for the Kuramoto model on complex networks. We obtained a number of necessary conditions for stability which restrict the position and weight of edges of the system's Coates graph. Among other results, this investigation provides a rigorous proof for the so-called positive-spanning-tree criterion \cite{Do2012} that states that stability in the Kuramoto model requires the existence of a spanning tree of edges associated with positive elemets of the system's Jacobian.

In network dynamics, system-level implications of meso-scale structures are just now intensively studied \cite{Fortunato,ChaosFocusIssue}. While efforts were long concentrated on the implications of local topological properties (such as the degree distribution), and global topological properties (such as the clustering coefficient), recent results \cite{Arenas,Nishikawa,Mori} indicate that also meso-scale properties can crucially affect synchronization. We believe that the results obtained here will be valuable for network dynamics, because they allow to extract meso-scale structural properties for the dynamics of networks. For a minor of given size $k$, already an application of Sylvester's criterion reveals necessary conditions for stability that restrict the structure of motifs containing $k$ vertices. The conditions obtained in this way are only necessary and not sufficient because instability could still occur due to instabilities in larger or smaller motifs. More complex necessary conditions for stability can be formulated on the meso-scale. E.g.\ for induced lines (Theorem \ref{thm-weightbound}), an obstruction to stability are lines with more then one negative edge or those with one negative edge whose weight is below a certain threshold.

In the future this approach may allow to attribute dynamical characteristics such as the onset of oscillations or the loss of synchronization to a specific meso-scale structure in the system. Such insights are presently actively sought in network dynamics \cite{GrossSayama2009}. Further, they may hold the key to important real world applications such as the operation of smart grids, future adaptive power grids, which will have to maintain synchronization in very different demand situations.

\end{document}